\documentclass[12pt]{amsart}
\usepackage{amssymb,amscd}
\usepackage{verbatim}

\usepackage{amsmath,amssymb,graphicx,mathrsfs}   
\usepackage[colorlinks=true,allcolors = blue]{hyperref} 

\let\frak\mathfrak

\def\>{\relax\ifmmode\mskip.666667\thinmuskip\relax\else\kern.111111em\fi}
\def\<{\relax\ifmmode\mskip-.333333\thinmuskip\relax\else\kern-.0555556em\fi}
\def\vsk#1>{\vskip#1\baselineskip}
\def\vv#1>{\vadjust{\vsk#1>}\ignorespaces}
\def\vvn#1>{\vadjust{\nobreak\vsk#1>\nobreak}\ignorespaces}

  \let\ssize\scriptstyle
\let\sssize\scriptscriptstyle

\let\Medskip\medskip
\def\medskip{\par\Medskip}
\let\Bigskip\bigskip
\def\bigskip{\par\Bigskip}

\let\Maketitle\maketitle
\def\maketitle{\Maketitle\thispagestyle{empty}\let\maketitle\empty}

\newtheorem{thm}{Theorem}[section]
\newtheorem{cor}[thm]{Corollary}

\theoremstyle{definition}                                  
\newtheorem{exmp}{Example}[section]

\numberwithin{equation}{section}

\theoremstyle{definition}
\newtheorem*{rem}{Remark}

\let\mc\mathcal
\let\nc\newcommand

\let\ka\kappa
\let\la\lambda

\let\phi\varphi

\let\Om\Omega

\let\der\partial

\let\ox\otimes

\let\geq\geqslant
\let\le\leqslant
\let\leq\leqslant

\let\on\operatorname
\let\bi\bibitem
\let\bs\boldsymbol

\def\C{{\mathbb C}}
\def\Z{{\mathbb Z}}

\def\F{{\mc F}}

\def\+#1{^{\{#1\}}}

\def\beq{\begin{equation}}
\def\eeq{\end{equation}}
\def\be{\begin{equation*}}
\def\ee{\end{equation*}}

\nc{\bea}{\begin{eqnarray*}}
\nc{\eea}{\end{eqnarray*}}
\nc{\bean}{\begin{eqnarray}}
\nc{\eean}{\end{eqnarray}}
\nc{\Ref}[1]{{\rm(\ref{#1})}}

\def\g{{\mathfrak g}}

\let\ga\gamma
\let\Ga\Gamma

\nc{\Il}{{\mc I_{\bs\la}}}
\nc{\bla}{{\bs\la}}
\nc{\Fla}{\F_\bla}
\nc{\tfl}{{T^*\Fla}}
\nc{\GL}{{GL_n(\C)}}
\nc{\GLC}{{GL_n(\C)\times\C^*}}

\let\sd s 

\def\ddk_#1{\kk_{#1}\<\>\frac\der{\der\<\>\kk_{#1}}}

\def\bul{\mathbin{\raise.2ex\hbox{$\sssize\bullet$}}}
\def\intt{\mathchoice
{\mathop{\raise.2ex\rlap{$\,\,\ssize\backslash$}{\intop}}\nolimits}
{\mathop{\raise.3ex\rlap{$\,\sssize\backslash$}{\intop}}\nolimits}
{\mathop{\raise.1ex\rlap{$\sssize\>\backslash$}{\intop}}\nolimits}
{\mathop{\rlap{$\sssize\<\>\backslash$}{\intop}}\nolimits}}

\let\kk q 
\let\cc c

\let\Ko K

\def\GZ/{Gelfand-Zetlin}
\def\KZ/{{\slshape KZ\/}}
\def\qKZ/{{\slshape qKZ\/}}
\def\XXX/{{\slshape XXX\/}}

\def\Sym{\on{Sym}}

\nc{\A}{{\mc C}}

\def\FF{{\mathbb F}}
\def\Sing{{\on{Sing}}}
\def\sll{{\frak{sl}}}

\def\slt{{\frak{sl}_2}}
\def\Ant{{\on{Ant}}}
\def\Ik{{\mc I_k}}
\def\Q{{\mathbb Q}}

\begin{document}

\hrule width0pt
\vsk->

\title[Solutions of KZ differential equations modulo  $p$]
{Solutions of KZ differential equations \\
modulo  $p$}

\author
[Vadim Schechtman and Alexander Varchenko]
{ Vadim Schechtman$\>^\circ$ and Alexander Varchenko$\>^\star$}

\maketitle

\begin{center}
{\it ${}^\circ$ Institut de Math\'ematiques de Toulouse\,--\,  Universit\'e Paul Sabatier\\ 118 Route de Narbonne,
31062 Toulouse, France \/}

\vsk.5>
{\it $^{\star}\<$Department of Mathematics, University
of North Carolina at Chapel Hill\\ Chapel Hill, NC 27599-3250, USA\/}

\end{center}

{\let\thefootnote\relax
\footnotetext{\vsk-.8>\noindent
$^\circ\<${\sl E\>-mail}:\enspace vadim.schechtman@math.univ-toulouse.fr
\\
$^\star\<${\sl E\>-mail}:\enspace anv@email.unc.edu\>,
supported in part by NSF grants DMS-1362924, DMS-1665239}}

\begin{abstract}
We construct polynomial solutions of the KZ differential equations over 
a finite field $\FF_p$ as analogs of hypergeometric solutions.

\end{abstract}

\bigskip

{\it Keywords}:\ KZ differential equations; multidimensional hypergeometric integrals; polynomial
solutions over finite fields.

{\it 2010 Mathematics Subject Classification}: 81R12 (11C08, 14H52)

\bigskip

\begin{center}
To Yu.I. Manin with admiration on the occasion of his 80th birthday
\end{center}

{\small \tableofcontents  }

\setcounter{footnote}{0}
\renewcommand{\thefootnote}{\arabic{footnote}}

\section{Introduction}
The KZ equations were discovered by physicists Vadim Knizhnik and Alexander Zamolodchikov  \cite{KZ}
to describe the differential equations for conformal blocks on sphere in the Wess-Zumino-Witten model of conformal
field theory. As I.M.\,Gelfand said, the KZ equations are remarkable 
differential equations discovered by physicists, defined in terms
of a Lie algebra and whose monodromy is described by the corresponding quantum group. It turned out that the KZ equations
are realized as suitable Gauss-Manin connections and its solutions are represented by multidimensional hypergeometric integrals,
see \cite {CF,DJMM,Mat,SV1,SV2,SV3}. The fact that certain integrals of closed differential forms over cycles
satisfy a linear  differential equation follows by Stokes' theorem from a suitable cohomological relation,
in which the  result of the
application of the corresponding differential operator to the integrand of an integral equals the differential of a form of one degree less.
Such cohomological relations for the KZ equations associated with Kac-Moody algebras were developed in \cite{SV3}. 

The goal of this paper is to construct polynomial solutions of the KZ differential equations over 
a finite field $\FF_p$ with $p$ elements, where $p$ is a prime number, as analogs of the hypergeometric solutions constructed in \cite{SV3}.
Our construction is based on the fact
that all cohomological relations described in \cite{SV3} are defined over $\Z$ and can be reduced 
modulo $p$. We learned how to construct  polynomial solutions in this situation out of hypergeometric solutions from the remarkable paper by
Yu.I.\,Manin \cite{Ma}, see a detailed exposition of Manin's idea in   Section ``Manin's Result: The Unity of Mathematics'' in the book \cite{Cl} by H.C.\,Clemens. 

In the remainder of the introduction we consider the  example of one-dimensional hypergeometric and $p$-hypergeometric integrals as an illustration of our constructions and results. The multidimensional case is considered in Sections \ref{2}-\ref{4}.

\subsection{Case of field $\C$}
\label{CaseC}

Let $\ka, m_1,\dots,m_n$ be nonzero complex numbers, $z=(z_1,\dots, z_n)\in\C^n, t\in\C$. Denote $|m|=m_1+\dots+m_n$.
Consider the {\it master function}
\bea
\Phi(t,z_1,\dots,z_n) = \prod_{1\leq a<b\leq n}(z_a-z_b)^{m_am_b/2\ka}\prod_{a=1}^n(t-z_a)^{-m_a/\ka}
\eea
and  the $n$-vector  
\bean
\label{Iga}
I^{(\ga)} (z)=(I_1(z),\dots,I_n(z)),
\eean
 where
\bean
\label{s}
I_j=\int  \Phi(t,z_1,\dots,z_n) \frac {dt}{t-z_j},\qquad j=1,\dots,n.
\eean
The integrals are over a closed (Pochhammer) curve $\ga$ in $\C-\{z_1,\dots,z_n\}$ on which one fixes a uni-valued branch 
of the master function to make the integral well-defined.
Starting from such a curve chosen for given $\{z_1,\dots,z_n\}$, the vector $I^{(\ga)}(z)$ can be analytically continued as a multivalued holomorphic function of $z$ to the complement in $\C^n$ to the union of the
diagonal  hyperplanes $z_i=z_j$.

\begin{thm}
\label{thm1.1}

 The vector $I^{(\ga)}(z)$ satisfies the algebraic equation
\bean
\label{le}
m_1I_1(z)+\dots+m_nI_n(z)=0
\eean
and  the differential KZ equations:
\bean\label{KZ}
 \frac{\partial I}{\partial z_i} \ = \
   {\frac 1\kappa} \sum_{j \ne i}
   \frac{\Omega_{i,j}}{z_i - z_j}  I ,
\qquad i = 1, \dots , n,
\eean
where
\[ \Omega_{i,j} \ = \ \begin{pmatrix}

                   & \vdots^i &  & \vdots^j &  \\

        {\scriptstyle i} \cdots & { \frac{(m_i -2)m_j}2 } & \cdots &
                 m_j & \cdots \\

                   & \vdots &  & \vdots &   \\

        {\scriptstyle j} \cdots & m_i & \cdots & \frac{m_i(m_j-2)}2 &
                 \cdots \\

                   & \vdots &  & \vdots &

                   \end{pmatrix} ,
                    \]
all other diagonal entries are $\frac{m_im_j}2$ and the
remaining off-diagonal entries are all zero.
\end{thm}

\begin{rem}
The vector $I^{(\ga)}(z)$ depends on the choice of the curve $\ga$. Different  curves give different solutions of the same KZ equations and all solutions of  equations \Ref{le} and \Ref{KZ} are obtained in this way, if $\kappa,m_1,\dots,m_n$ are generic.
\end{rem}

\begin{rem}
The differential equations \Ref{KZ} are the KZ differential equations with parameter $\ka$ associated with the Lie algebra $\sll_2$ and the singular weight subspace of weight $|m|-2$ of the tensor product of $\sll_2$-modules with highest weights
$m_1,\dots,m_n$, see Section \ref {2}.
\end{rem}

\begin{rem}
The KZ equations define a flat connection over the complement in $\C^n$ to the union of all diagonal hyperplanes,
\bean
\label{flat}
 \left[\frac{\partial }{\partial z_i} -
   {\frac 1\kappa} \sum_{j \ne i}
   \frac{\Omega_{i,j}}{z_i - z_j},\
  \frac{\partial }{\partial z_k} -
   {\frac 1\kappa} \sum_{j \ne k}
   \frac{\Omega_{k,j}}{z_k - z_j}\right]=0
\eean
for all $j,k$.
\end{rem}

Theorem \ref{thm1.1} is a classical statement probably known in 19th century. 
Much more general algebraic and differential equations satisfied by analogous multidimensional hypergeometric integrals were considered in \cite{SV3}.  Theorem \ref{thm1.1} is discussed as an example in  \cite[Section 1.1]{V2}.

Below we give a proof of Theorem \ref{thm1.1}. A modification of this proof  in Section \ref{Case F}
will produce for us polynomial solutions of the equations
\Ref{le} and \Ref{KZ} modulo a prime $p$.

\medskip
\noindent
{\it Proof of Theorem \ref{thm1.1}.}
Equations \Ref{le} and \Ref{KZ}  are implied by the following cohomological identities.
We have
\bean
\label{i1}
\frac{-m_1}\ka\Phi(t,z)\frac {dt}{t-z_1} + \dots +\frac{-m_1}\ka\Phi(t,z) \frac {dt}{t-z_n} = d_t \Phi(t,z),
\eean
where $d_t$ denotes the differential with respect to the variable $t$.
This identity and Stokes' theorem imply equation \Ref{le}.

Denote
\bean
\label{V}
V(t,x) = \Big(\frac {dt}{t-z_1}, \dots, \frac {dt}{t-z_n}\Big).
\eean
For any $i=1,\dots,n$,  let $W^i(t,z)$ be the vector of  $(0,\dots,0,\frac{-1}{t-z_i},0,\dots,0)$ with nonzero element at the $i$-th place. Then
\bean
\label{ci}
\Big(\frac{\der I}{\der z_i}-\frac1\ka
\sum_{j\ne i} \frac {\Omega_{i,j}}{z_i-z_j} \Big)\Phi (t,z) V(t,x) = d_t(\Phi(t,z) W^i(t,z)).
\eean
The proof of this identity is straightforward. Much more general identities of this type see in
 \cite[Lemmas 7.5.5 and 7.5.7]{SV3}, cf. identities in Section \ref{sec2.4}.
 
 Identity \Ref{ci} and Stokes' theorem imply the KZ equation \Ref{KZ}.
\qed

\begin{exmp}
\label{ex1}
Let $\ka=2$,  $n=3$, $m_1=m_2=m_3=1$. Then $I^{(\ga)}(z)=(I_1(z),I_2(z),I_3(z))$,  where
\bean
\label{I3}
I_j(z) = \prod_{1\leq a< b\leq 3} \!\!\! \sqrt[4]{z_a-z_b}\
\int_{\ga(z)} \frac 1{\sqrt{(t-z_1)(t-z_2)(t-z_3)}}\frac {dt}{t-z_j}.
\eean
In this case, the curve $\ga(z)$ may be thought of as a closed path on the elliptic curve
\bea
y^2=(t-z_1)(t-z_2)(t-z_3).
\eea
Each of these integrals is an elliptic integral. Such an integral is
 a branch of analytic continuation of a suitable Euler hypergeometric function up to change of variables.


\end{exmp}

\subsection{Case of field $\FF_p$}
\label{Case F}

Let $\ka, m_1,\dots,m_n$ be positive integers. 
Let $p>2$ be a prime number, $p\nmid\ka$.
The algebraic equation \Ref{le} and the differential KZ equations \Ref{KZ}
are well-defined when reduced modulo $p$. The reduction of the KZ equations
satisfies the flatness condition \Ref{flat}.
We  construct solutions of  equations \Ref{le} and \Ref{KZ}
with values in $(\FF_p[z])^n$. Notice that the space of such solutions is a module over the ring 
$\FF_p[z_1^p,\dots,z_n^p]$ since $\frac{\der z_i^p}{\der z_j} =0$.

Choose positive integers $M_a$ for $a=1,\dots,n$ and
$M_{a,b}$ for $1\le a<b\leq n$ such that 
\bea
M_a\equiv -\frac{m_a}\ka,
\qquad 
M_{a,b}\equiv \frac{m_am_b}{2\ka} 
\qquad 
(\on{mod} \,p).
\eea
  That means that we project $m_a, \ka, 2$ to $\FF_p$, calculate $-\frac{m_a}\ka, \frac{m_am_b}{2\ka}$ in $\FF_p$
  and then choose positive integers $M_a, M_{a,b}$ satisfying these equations.
  
      Fix an integer $q$. Consider the {\it master polynomial}
\bea
\Phi^{(p)}(t,z) = \prod_{1\leq a<b\leq n}(z_a-z_b)^{M_{a,b}}\prod_{a=1}^n(t-z_a)^{M_a},
\eea
and the Taylor expansion with respect to the variable $t$ of the vector of polynomials
\bea
\Phi^{(p)}(t,z)\Big(\frac 1{t-z_1}, \dots,\frac 1{t-z_n}\Big)\,=\, \sum_i \bar I^{(i)}(z,q) \,(t-q)^i,
\eea
where the $\bar I^{(i)}(z,q)$ are $n$-vectors of polynomials in $z$ with integer coefficients.
Let $I^{(i)}(z,q) \in (\FF_p[z])^n $ be the canonical projection of
$\bar I^{(i)}(z,q)$.

\begin{thm} 
\label{thm 1.1}
For any integer $q$ and positive integer $l$, the vector of polynomials
\\
 $ I^{(lp-1)}(z,q)$ 
satisfies equations \Ref{le} and \Ref{KZ}.

\end{thm}

The parameters $q$ and $lp-1$ are analogs of cycles $\ga$ in Section \ref{CaseC}.

\begin{proof}

To prove that $I^{(lp-1)}(z,q)$ satisfies \Ref{le} and \Ref{KZ} we consider the Taylor expansions at $t=q$ of both sides of equations
\Ref{i1} and 
\Ref{ci}, divide them by $dt$, and then project the  coefficients of $(t-q)^{lp-1}$
  to $\FF_p[z]$. The projections of  the right-hand sides equal 
zero since $d(t^{lp})/dt=lpt^{lp-1}\equiv 0$ (mod $p$).
\end{proof}

\begin{exmp}
\label{ex2}
Let  $\ka=2$, $m_1=\dots=m_n=1$, cf. Example \ref{ex1}. 
Given $p>2$ choose the master polynomial
\bean
\Phi^{(p)}(t,z) = \prod_{1\leq a<b\leq n}(z_a-z_b)^{\frac{(p+1)^2}4}\prod_{s=1}^n(t-z_s)^{\frac{p-1}2}.
\eean
Consider the Taylor expansion
\bean
\label{te}
\prod_{s=1}^n(t-z_s)^{\frac{p-1}2} \Big(\frac 1{t-z_1}, \dots, \frac 1{t-z_n}\Big)
= \sum_i \bar c^i(z) t^i,
\eean
where $\bar c^i=(\bar c^i_1, \dots, \bar c^i_n)$. Let $c^i$ be the projection of $\bar c^i$ to $(\FF_p[z])^n$.
Then the vector of polynomials
\bean
\label{3sol}  
\phantom{aaaaaa}
I(z)=
(I_1(z),\dots,I_n(z))= \!\!
\prod_{1\leq a<b\leq n}\!\!
(z_a-z_b)^{\frac{(p+1)^2}4} \left(c^{p-1}_1(z), \dots ,  c^{p-1}_n(z)\right)
\eean
is a solution of the KZ differential equations over $\FF_p[z]$ 
and $I_1(z)+\dots+I_n(z)=0$.

\end{exmp}

\begin{exmp}
\label{ex1.2}

Let $\ka=2$, $m_1=\dots=m_n=1$, $p=3$. We have  
\bea
\Om_{i,j} (I_1,\dots,I_n)&=&\frac12(I_1,\dots,I_{i-1}, -I_i+2I_j, I_{i+1},\dots, I_{j-1}, 2I_i-I_j, I_{j+1},\dots,I_n)
\\
&\equiv & (-I_1,\dots,-I_{i-1}, I_i+I_j, -I_{i+1},\dots, -I_{j-1}, I_i+I_j, -I_{j+1},\dots,-I_n) 
\eea
(mod $3$).  Equation \Ref{le} has the form
$I_1(z)+\dots+I_n(z)=0$.
We may choose the master polynomial 
\bea
\Phi^{(p=3)}(t,z) = \prod_{1\leq a<b\leq n}(z_a-z_b)\prod_{s=1}^n(t-z_s).
\eea
  Choose a nonnegative integer $l$. Then the vector $I(z,q):= I^{(3l-1)}(z,q)=(I_1(z,q),$ \dots, $I_n(z,q))$ of Theorem
\ref{thm 1.1} has  coordinates
\bean
\label{sol}
I_j(z,q)=  \Big(\prod_{1\leq a<b\leq n}(z_a-z_b) \Big)
\sum_{1\leq i_1<\dots<i_{n-3-3l}\leq n, \atop
j\notin \{i_1,\dots,i_{n-3-3l}\}}\prod_{a=1}^{n-3-3l}(q-z_{i_a})
\eean
and is a solution  of \Ref{le} and \Ref{KZ} with values in $(\FF_3[z])^n$ for any $q=0,1,2$. 
 Expanding these solutions into polynomials 
 homogeneous in $z$ we obtain solutions in homogeneous polynomials,
 which stabilize with respect to $n$ as follows. 
 The vector $I^{[r]}(z)=(I_1^{[r]}(z),\dots,I_n^{[r]}(z))$, with coordinates
\bean
\label{sols}
I^{[r]}_j(z)=  \Big(\prod_{1\leq a<b\leq n}(z_a-z_b) \Big)
\sum_{1\leq i_1<\dots<i_{r}\leq n, \atop
j\notin \{i_1,\dots,i_{r}\}}\prod_{a=1}^{r}z_{i_a},
\eean
is a solution  of \Ref{le} and \Ref{KZ} with values in $(\FF_3[z])^n$ if $r\equiv n$ (mod  $3$)
and $r<n$.
Thus, the vector $I^{[0]}(z)$, with coordinates 
 \bean
\label{sols0}
I_j^{[0]}(z)=  \prod_{1\leq a<b\leq n}(z_a-z_b), 
\eean
is a solution  with values in $(\FF_3[z])^n$ for $n\equiv 0$ (mod $3$); the vector $I^{(1)}(z)$,
with coordinates
\bean
\label{sols1}
I_j^{[1]}(z)=  \Big(\prod_{1\leq a<b\leq n}(z_a-z_b) \Big)
\sum_{1 \leq i \leq n, \, i\ne j} z_i ,
\eean
is a solution  
for $n\equiv 1$ (mod $3$) and so on.  Note that the sum in \Ref{sols} is the $m$-th elementary symmetric function in
$z_1,\dots, \widehat{z_j}, \dots,z_n$.

Solutions provided by Theorem \ref{thm 1.1} depend on parameters $q$, $lp-1$. In this example all solutions $I^{[r]}(z)$ can be 
obtained by putting $q=0$ and varying $lp-1$ only. 

\end{exmp}

\subsection{Relation of polynomial solutions to integrals over $\FF_p$} For a polynomial $F(t)\in\FF_p[t]$
define the integral 
\bea
\int_{\FF_p} F(t) := \sum_{t\in\FF_p}F(t).
\eea
Recall that 
\bean
\label{sum=}
\on{the\ sum} \ \ \sum_{t\in\FF_p}t^i\ \  \on{equals}\ \  -1 \ \on{if}\ (p-1)\big| i \
\on{ and\ equals\ zero\ otherwise.}
\eean

\begin{thm}
\label{thm int}  Fix $x_1,\dots,x_n, q\in \FF_p$.
Consider the vector of polynomials 
\bea
F(t,x_1,\dots,x_n):=\Phi^{(p)}(t,x_1,\dots,x_n)\Big(\frac 1{t-x_1}, \dots,\frac 1{t-x_n}\Big)
\ \in\  \FF_p[t]
\eea
 of Section \ref{Case F}.  Assume that
$\deg_t F(t,x_1,\dots,x_n) < 2p-2$. Consider the polynomial solution
 $I^{(p-1)}(z_1,\dots,z_n, q)$ of equations \Ref{le} and \Ref{KZ} defined in front of Theorem
\ref{thm 1.1}. 
Then 
\bean
\label{INT}
I^{(p-1)}(x_1,\dots,x_n, q) = -\int_{\FF_p}F(t,x_1,\dots,x_n) .
\eean
\end{thm}
This integral is a $p$-analog of the hypergeometric integral \Ref{s}.

\begin{proof} Consider the Taylor expansion  $F(t,x_1,\dots,x_n) =  \sum_{i=0}^{2p-3} I^{(i)}(x_1,\dots,x_n,q)(t-q)^i$. By formula
\Ref{sum=}, we have $\sum_{t\in\FF_p}F(t,x_1,\dots,x_n)=-I^{(p-1)}(x_1,\dots,x_n, q)$.
\end{proof}

\begin{exmp}
\label{exI}
Given $\ka$, $n$, $m_1=\dots=m_n=1$, assume that $n\leq 2\ka$ and $\ka\big| (p-1)$. Then
$F(t) = \prod_{a<b}(z_a-z_b)^{M_{a,b}}\prod_{s=1}^n(t-x_s)^{\frac{p-1}\ka}\big(\frac1{t-x_1},
\dots,\frac1{t-x_n}\big)$ and $\deg_tF(t)<2p-2$.

\end{exmp}

\subsection{Relation of solutions to  curves over $\FF_p$}  
\label{sec1.3}

\begin{exmp}
\label{ex131}

Let $x_1,x_2,x_3\in\FF_p$. Let
$\Ga(x_1,x_2,x_3)$ be
the projective closure of the affine  curve
\bean
\label{ec}
y^2 = (t-x_1)(t-x_2)(t-x_3)
\eean
over $\FF_p$. 
For a rational function $h : \Ga(x_1,x_2,x_3)\to \FF_p$  define the integral  
\bean
\label{int}
\int_{\Ga(x_1,x_2,x_3)} h = {\sum'_{P\in\Ga(x_1,x_2,x_3)}}\,h(P),
\eean
as the sum over all points $P\in\Ga(x_1,x_2,x_3)$, where $h(P)$  is defined.

\begin{thm}
\label{thm pts} 
Let $p>2$ be a prime.
Let $\left( c^{p-1}_1(x_1,x_2,x_3),  c^{p-1}_2(x_1,x_2,x_3),  c^{p-1}_3(x_1,x_2,x_3)\right)$ 
be the vector of polynomials appearing in the solution 
\Ref{3sol} of the KZ equations of Example \ref{ex2} for $n=3$. Then 
\bean
\label{pts=sol} 
\int_{\Ga(x_1,x_2,x_3)}\frac 1{t-x_j}
 =   - \, c^{p-1}_j(x_1,x_2,x_3),
\qquad j=1,2,3.
\eean
\end{thm}

\begin{rem}
Theorems  \ref{thm 1.1}  and \ref{thm pts}  say that the integrals $\int_{\Ga(x_1,x_2,x_3)}\frac 1{t-x_j}$
are polynomials in $x_1,x_2,x_3\in\FF_p$ and the triple of polynomials
\bea
I(x_1,x_2,x_3)=
 \!\!
\prod_{1\leq a<b\leq 3}\!\!
(x_a-x_b)^{\frac{(p+1)^2}4}
 \left( \int_{\Ga(x_1,x_2,x_3)}\frac 1{t-x_1},  \int_{\Ga(x_1,x_2,x_3)}\frac 1{t-x_2},  \int_{\Ga(x_1,x_2,x_3)}
 \frac 1{t-x_3}\right)
\eea
in these discrete variables
satisfies the KZ differential equations! Cf. Example \ref{ex1}.

\end{rem}

\noindent
{\it Proof of Theorem \ref{thm pts}.}
The proof is analogous to the reasoning in \cite[Section 2]{Ma} and \cite{Cl}. 
The value of $1/(t-x_j)$ at the infinite point of $\Ga(x_1,x_2,x_3)$ equals zero. 
It is easy to see that 
\bea
\int_{\Ga(x_1,x_2,x_3)}\frac 1{t-x_j}
& = &
\sum_{t\in\FF_p,\, t\ne x_j} \frac 1{t-x_j} 
+\sum_{t\in\FF_p} \frac {1}
{t-x_j}\prod_{s=1}^3(t-x_s)^{\frac{p-1}2}
\\
&=&  \sum_{t\in\FF_p} (t-x_j)^{p-2} 
 + \sum_{t\in\FF_p}  \sum_i  c^i_j(x_1,x_2,x_3) t^i
=   -  c^{p-1}_j(x_1,x_2,x_3),
\eea
where the last equality is by formula \Ref{sum=}.
\qed

\begin{rem}
In \cite[Section 2]{Ma} and in \cite{Cl}, an
equation analogous to
\Ref{pts=sol}   is considered, where the left-hand side is the number of points
on $\Ga(x_1,x_2,x_3)$ over $\FF_p$ and the right-hand side is the reduction modulo $p$ of a solution of a 
second order Euler hypergeometric differential equation. Notice that the number of points on $\Ga(x_1,x_2,x_3)$
 is the discrete integral over $\Ga(x_1,x_2,x_3)$ of the 
constant function $h=1$.  See details in Section ``Manin's Result: The Unity of Mathematics'' in \cite{Cl}.

\end{rem}

\end{exmp}

\begin{exmp}
\label{ex132}  This example is a variant of  Example \ref{ex131}.

Let $x_1,x_2,x_3,x_4\in\FF_p$. Let
$\Ga(x_1,x_2,x_3,x_4)$ be
the projective closure of the affine  curve
\bean
\label{ec}
y^2 = (t-x_1)(t-x_2)(t-x_3)(t-x_4)
\eean
over $\FF_p$. 

Let $p>3$ be a prime.
Let 
\bea
\left( c^{p-1}_1(x_1,x_2,x_3,x_4),  c^{p-1}_2(x_1,x_2,x_3,x_4),  c^{p-1}_3(x_1,x_2,x_3,x_4),
 c^{p-1}_4(x_1,x_2,x_3,x_4)\right)
\eea
be the vector of polynomials appearing in the solution 
\Ref{3sol} of the KZ equations of Example \ref{ex2} for $n=4$. Then 
\bean
\label{pts=2sol} 
\int_{\Ga(x_1,x_2,x_3,x_4)}\frac 1{t-x_j}
& = &  - \,c^{p-1}_j(x_1,x_2,x_3,x_4),
\qquad
j=1,2,3,4.
\eean

\end{exmp}

\begin{exmp}
\label{ex133}
Let  $\ka=3$, $n=3$, $m_1=m_2=m_3=2$. Assume that $3\big| (p-1)$.
Choose the master polynomial 
\bea
\Phi^{(p)}(t,z)=\prod_{1\leq a<b\leq 3}(z_a-z_b)^{\frac{p+2}3}\prod_{s=1}^3(t-z_s)^{2\frac{p-1}3}.
\eea
Consider the Taylor expansion
\bean
\label{te}
\prod_{s=1}^3(t-z_s)^{2\frac{p-1}3} \Big(\frac 1{t-z_1}, \frac 1{t-z_2}, \frac 1{t-z_3}\Big)
= \sum_i \bar c^i(z_1,z_2,z_3) t^i,
\eean
where $\bar c^i=(\bar  c^i_1, \bar c^i_2, \bar c^i_3)$. Let $ c^i$ be the projection of $\bar c^i$ to $(\FF_p[z])^3$.
Then the vector 
\bean
\label{4sol}  
\phantom{aaaaaa}
I(z)=
(I_1(z),I_2(z),I_3(z))= \!\!
\prod_{1\leq a<b\leq 3}\!\!
(z_a-z_b)^{\frac{p+2}3} \left( c^{p-1}_1(z),  c^{p-1}_2(z) ,  c^{p-1}_3(z)\right)
\eean
is a solution of the corresponding  KZ differential equations over $\FF_p[z]$ 
and $I_1(z)+I_2(z)+I_3(z)=0$.

For distinct $x_1,x_2,x_3\in\FF_p$ let
$\Ga(x_1,x_2,x_3)$ be
the projective closure of the affine 
\bean
\label{ec}
y^3 = (t-x_1)(t-x_2)(t-x_3)
\eean
over $\FF_p$. The curve has 3 points at infinity.

\begin{thm}
\label{thm pts4} 
Let $p$ be a prime such that $3\big|(p-1)$.
Let 
\bea
\left( c^{p-1}_1(x_1,x_2,x_3),  c^{p-1}_2(x_1,x_2,x_3),  c^{p-1}_3(x_1,x_2,x_3)\right)
\eea
be the vector of polynomials appearing in the solution 
\Ref{4sol} of the KZ equations. Then for $j=1,2,3$ we have
\bean
\label{pts=3sol} 
\int_{\Ga(x_1,x_2,x_3)}\frac 1{t-x_j}
& = &  - \, c^{p-1}_j(x_1,x_2,x_3).
\eean
\end{thm}

\begin{proof}

The value of $1/(t-x_j)$ at infinite points of $\Ga$ equals zero. 
It is easy to see that 
\bean
\label{3lines}
&&
\int_{\Ga(x_1,x_2,x_3)}\frac 1{t-x_j}
 =
\sum_{t\in\FF_p,\, t\ne x_j} \frac 1{t-x_j} 
\\
\notag
&&
\phantom{aa}
+\sum_{t\in\FF_p} \frac {1}
{t-x_j}\prod_{s=1}^3(t-x_s)^{\frac{p-1}3}
+\sum_{t\in\FF_p} \frac {1}
{t-x_j}\prod_{s=1}^3(t-x_s)^{2\frac{p-1}3}
\\
\notag
&&
\phantom{aaaaa}
=  \sum_{t\in\FF_p} (t-x_j)^{p-2} 
 + \sum_{t\in\FF_p}  \sum_i  c^i_j(x_1,x_2,x_3) t^i
=    -  c^{p-1}_j(x_1,x_2,x_3). 
\eean
Notice that $\sum_{t\in\FF_p} \frac {1}
{t-x_j}\prod_{s=1}^3(t-x_s)^{\frac{p-1}3}=0$ since the polynomial under the sum is of degree 
$p-2$ which is less than $p-1$. The last equality in \Ref{3lines} is by formula \Ref{sum=}.
\end{proof}

\end{exmp}

\begin{exmp}
\label{ex135}
Let  $\ka=3$, $n=3$, $m_1=m_2=1$, $m_3=2$. Assume that $3$  divides $p-1$.
Choose the master polynomial 
\bea
\Phi^{(p)}(t,z)=(z_1-z_2)^{\frac{5p+1}6}(z_1-z_3)^{\frac{2p+1}3}(z_2-z_3)^{\frac{2p+1}3}
(t-z_1)^{\frac{p-1}3}(t-z_2)^{\frac{p-1}3}(t-z_3)^{2\frac{p-1}3}.
\eea
Consider the Taylor expansion
\bean
\label{te2}
\phantom{aaa}
(t-z_1)^{\frac{p-1}3}(t-z_2)^{\frac{p-1}3}(t-z_3)^{2\frac{p-1}3} \Big(\frac 1{t-z_1}, \frac 1{t-z_2}, \frac 1{t-z_3}\Big)
= \sum_i \bar b^i(z_1,z_2,z_3) t^i,
\eean
where $\bar b^i=(\bar  b^i_1, \bar b^i_2, \bar b^i_3)$. Let $ b^i$ be the projection of $\bar b^i$ to $(\FF_p[z])^3$.
Then the vector 
\bean
\label{44sol}  
\phantom{aaaaaa}
I(z)=
(z_1-z_2)^{\frac{5p+1}6}(z_1-z_3)^{\frac{2p+1}3}(z_2-z_3)^{\frac{2p+1}3}
\left( b^{p-1}_1(z),  b^{p-1}_2(z) ,  b^{p-1}_3(z)\right)
\eean
is a solution of the corresponding KZ differential equations over $\FF_p[z]$ 
and $I_1(z)+I_2(z)+2I_3(z)=0$.

\medskip

Similarly
let  $\ka=3$, $n=3$,  $m_1=m_2=2$, $m_3=1$. Assume that $3$  divides $p-1$.
Choose the master polynomial 
\bea
\Phi^{(p)}(t,z)=(z_1-z_2)^{\frac{p+2}3}(z_1-z_3)^{\frac{2p+1}3}(z_2-z_3)^{\frac{2p+1}3}
(t-z_1)^{2\frac{p-1}3}(t-z_2)^{2\frac{p-1}3}(t-z_3)^{\frac{p-1}3}.
\eea
Consider the Taylor expansion
\bean
\label{te3}
\phantom{aaa}
(t-z_1)^{2\frac{p-1}3}(t-z_2)^{2\frac{p-1}3}(t-z_3)^{\frac{p-1}3} \Big(\frac 1{t-z_1}, \frac 1{t-z_2}, \frac 1{t-z_3}\Big)
= \sum_i \bar c^i(z_1,z_2,z_3) t^i,
\eean
where $\bar c^i=(\bar  c^i_1, \bar c^i_2, \bar c^i_3)$. Let $ c^i$ be the projection of $\bar c^i$ to $(\FF_p[z])^3$.
Then the vector 
\bean
\label{5sol}  
\phantom{aaaaaa}
I(z)=
(z_1-z_2)^{\frac{p+2}3}(z_1-z_3)^{\frac{2p+1}3}(z_2-z_3)^{\frac{2p+1}3}
\left( c^{p-1}_1(z),  c^{p-1}_2(z) ,  c^{p-1}_3(z)\right)
\eean
is a solution of the corresponding KZ differential equations over $\FF_p[z]$ 
and $2I_1(z)+2I_2(z)+I_3(z)=0$.

\medskip

For distinct $x_1,x_2,x_3\in\FF_p$ let
$\Ga(x_1,x_2,x_3)$ be
the projective closure of the affine curve
\bean
\label{ec}
y^3 = (t-x_1)(t-x_2)(t-x_3)^2
\eean
over $\FF_p$. The curve has genus 2 and one point at infinity.

\begin{thm}
\label{thm pts5} 
Let $p$ be a prime such that 3 divides $p-1$.
Let 
\bea
\left( b^{p-1}_1(x_1,x_2,x_3),  b^{p-1}_2(x_1,x_2,x_3),  b^{p-1}_3(x_1,x_2,x_3)\right)
\eea
be the vector of polynomials appearing in the solution 
\Ref{44sol} of the KZ equations with  $n=3$, $\ka=3$, $m_1=m_2=1$, $m_3=2$.
Let 
\bea
\left( c^{p-1}_1(x_1,x_2,x_3),  c^{p-1}_2(x_1,x_2,x_3),  c^{p-1}_3(x_1,x_2,x_3)\right)
\eea
be the vector of polynomials appearing in the solution 
\Ref{5sol} of the KZ equations with  $n=3$, $\ka=3$, $m_1=m_2=2$, $m_3=1$.
Then for $j=1,2,3$ we have
\bean
\label{pts=5sol} 
\int_{\Ga(x_1,x_2,x_3)}\frac 1{t-x_j}
& = &  - \, b^{p-1}_j(x_1,x_2,x_3) - \, c^{p-1}_j(x_1,x_2,x_3).
\eean
\end{thm}

\begin{proof}

The value of $1/(t-x_j)$ at infinite points of $\Ga$ equals zero. 
It is easy to see that 
\bea
&&
\int_{\Ga(x_1,x_2,x_3)}\frac 1{t-x_j}
 = 
\sum_{t\in\FF_p,\, t\ne x_j} \frac 1{t-x_j} 
+\sum_{t\in\FF_p} \frac 1{t-x_j}(t-z_1)^{\frac{p-1}3}(t-z_2)^{1\frac{p-1}3}(t-z_3)^{2\frac{p-1}3}
\\
&&
\phantom{a}
+\sum_{t\in\FF_p} \frac 1{t-x_j}(t-z_1)^{2\frac{p-1}3}(t-z_2)^{2\frac{p-1}3}(t-z_3)^{4\frac{p-1}3}
=  \sum_{t\in\FF_p} (t-x_j)^{p-2} 
\\
&&
\phantom{aa}
+ \sum_{t\in\FF_p}  \sum_i  b^i_j(x_1,x_2,x_3) t^i
+\sum_{t\in\FF_p} \frac 1{t-x_j}(t-z_1)^{2\frac{p-1}3}(t-z_2)^{2\frac{p-1}3}(t-z_3)^{\frac{p-1}3}
\\
&&
\phantom{aaa}
=
-\,b^{p-1}_j(x_1,x_2,x_3) + \sum_{t\in\FF_p}  \sum_i  c^i_j(x_1,x_2,x_3) t^i
=
- b^{p-1}_j(x_1,x_2,x_3) - c^{p-1}_j(x_1,x_2,x_3).
\eea
\end{proof}

\end{exmp}

\subsection{Resonances over $\C$ and $\FF_p$}
Under assumptions of Section \ref{CaseC} assume that
\bean
\label{res1}
m_1+\dots+m_n=\ka.
\eean
Then the vector $I^{(\ga)}(z)$, defined in \Ref{Iga},
 in addition to the algebraic equation \Ref{le} and differential equations \Ref{KZ} satisfies  the algebraic equation
\bean
\label{res2}
z_1m_1I_1(z)+\dots+z_nm_nI_n(z) = 0.
\eean
Equation \Ref{res2} follows from the  cohomological relation:
\bean
\label{crr}
d_t  (t\Phi)
& = & 
 \Phi dt  -  \Phi \sum_{j=1}^n 
\frac{m_j}\kappa\,\frac{t-z_j+z_j}{t-z_j}  dt
\\
\notag
&= &
\Big( 1 - \sum_{j=1}^n\, \frac{m_j}\ka \Big)\Phi dt
 -  \sum_{j=1}^n  z_j  \frac{m_j} {\ka}\Phi
\frac{dt}{t-z_j} .
\eean

 Relation \Ref{res2} manifests  resonances in conformal field theory,
  where  solutions of KZ equations
represent conformal blocks and conformal blocks satisfy algebraic equations  
analogous to \Ref{res2}, see \cite{FSV1, FSV2}, Section
3.6.2 in \cite{V2}.
 In conformal field theory the numbers
$m_1, $ $\dots ,$ $ m_n, $\ $\kappa$ are natural numbers.
 In that case the master function
$\Phi(t,z)$ is an algebraic function and the hypergeometric integrals
become integrals of algebraic forms over cycles lying on suitable algebraic varieties.
 The monodromy of the hypergeometric integrals $I^{(\ga)}(z)$ in that case 
was studied in Sections 13 and 14 of \cite{V1}.

\medskip
Relation \Ref{res2} has an analog over $\FF_p$.

\begin{thm}
Under assumptions of Theorem \ref{thm 1.1} let $ I^{(lp-1)}(z,q)\in\FF_p[z]^n$ be
the polynomial solution of equations \Ref{le} and \Ref{KZ} described in
Theorem  \ref{thm 1.1}.  Assume that
\bean
\label{res1p}
M_1+\dots+M_n\equiv -1 \ \  (\on{mod}\ p).
\eean
Then 
\bean
\label{res2p}
z_1M_1I_1(z)+\dots+z_nM_nI_n(z) = 0.
\eean
\end{thm}

\begin{proof}
The theorem follows from \Ref{crr} similarly to the proof of Theorem \ref{thm 1.1}.
Namely, we consider the Taylor expansions at $t=q$ of both sides of equation
\Ref{crr}, divide them by $dt$, and then project the  coefficients of $(t-q)^{lp-1}$
  to $\FF_p[z]$. The projection coming from $d_t  (t\Phi)$  equals 
zero since $d(t^{lp})/dt=lpt^{lp-1}\equiv 0$ (mod $p$). The projection coming from
$\big( 1 - \sum_{j=1}^n\, \frac{m_j}\ka \big)\Phi dt$ equals zero by \Ref{res1p}.
The projection coming from $-  \sum_{j=1}^n  z_j  \frac{m_j} {\ka}\Phi
\frac{dt}{t-z_j}$ gives \Ref{res2p}.
\end{proof}

\begin{exmp}
\label{ex1.4}

Let $\ka=2$, $m_1=\dots=m_n=1$, $p=3$,  $M_1=\dots=M_n=1$,
\bea
\Phi^{(p=3)}(t,z) = \prod_{1\leq a<b\leq n}(z_a-z_b)\prod_{s=1}^n(t-z_s)
\eea
as in  Example \ref{ex1.2}. 
Let $n  \equiv 2$ (mod $3$), then $M_1+\dots+M_n\equiv -1$ (mod $3$).
Choose a positive integer $r$, such that $r\equiv n$ (mod $3$) and $r<n$.
Then the   vector $I^{[r]}(z)$ given by \Ref{sols} satisfies equations \Ref{le}, \Ref{KZ},  and
\bea
z_1I^{[r]}_1(z)+\dots + z_nI^{[r]}_n(z) \equiv 0\ \ (\on{mod}\ 3).
\eea

\end{exmp}

\subsection{Exposition of material}
In Section \ref{2} we describe the hypergeometric solutions of the KZ equations associated with $\slt$ and explain their reduction  to polynomial solutions over  $\FF_p$. In Section \ref{3} we describe the resonance relations for
$\slt$ conformal blocks and construct their reduction over $\FF_p$. In Section \ref{4} we explain how the results of Section \ref{2} and \ref{3} are extended to the KZ equations associated with simple Lie algebras.

\medskip

This article was inspired by lectures on hypergeometric motives by Fernando Rodriguez-Villegas 
in May 2017 at MPI in Bonn. The authors thank him for stimulating discussions.
We were also motivated by the classical paper by Yu.I.\,Manin \cite{Ma}, from which we learned how to construct solutions
of differential equations over $\FF_p$ from cohomological relations between algebraic differential forms.
The authors thank A.\,Buium, Yu.I.\,Manin, and W. Zudilin for useful discussions
 and the referee for
 comments and suggestions contributed to improving  the presentation.

The article was conceived during the Summer 2017 Trimester
program ``K-Theory and Related Fields" of the Hausdorff Institute for Mathematics
(HIM), Bonn. The authors are thankful to HIM for stimulating atmosphere and
working conditions.  The first author is grateful to Max Planck Institute for Mathematics
 for hospitality 
during a visit in June 2017.

\section{$\sll_2$ KZ equations}
\label{2}

In this section we describe solutions of the KZ equations associated with the Lie algebra
$\slt$. The
solutions to the KZ equations over $\C$ in the form of multidimensional hypergeometric integrals
are known since the end of 1980s. The polynomial solutions of the KZ equations over $\FF_p$ in the form
of $\FF_p$-analogs of the multidimensional hypergeometric integrals are new.

\subsection{$\slt$ KZ equations}
Let $e,f,h$ be standard basis of the complex Lie algeba $\slt$ with $[e,f]=h$, $[h,e]=2e$, $[h,f]=-2f$.
The element  
\bean
\label{Casimir}
\Omega =  e \otimes f + f \otimes e +
                       \frac{1}{2} h \otimes h\  \in\  \slt \ox\slt
                       \eean
 is called the Casimir element.  Given $n$, for $1\leq i<j\leq n$ let $\Om^{(i,j)} \in (U(\slt))^{\ox n}$
 be the element equal to
 $\Omega$ in the $i$-th and $j$-th factors and to 1 in the other factors.
 For $i=1,\dots,n$ and distinct $z_1,\dots,z_n\in \C$ introduce
 \bean
 \label{Ham}
 H_i(z_1,\dots,z_n)= \sum_{j\ne i}\frac{\Om^{(i,j)}}{z_i-z_j} \ \in \ (U(\slt))^{\ox n},
 \eean
 the Gaudin Hamiltonians.
 For any $\ka\in\C^\times$ and any $i,k$, we have
 \bean
 \label{Flat}
 \left[\frac{\der}{\der z_i} - \frac1\ka H_i (z_1,\dots,z_n), \frac{\der}{\der z_k} - \frac1\ka H_k (z_1,\dots,z_n) \right]
=0,
\eean 
and for any $x\in\slt$ and $i$    we have
\bean
\label{h invar}
[H_i(z_1,\dots,z_n), x\ox 1\ox\dots\ox 1+\dots + 1\ox\dots\ox 1\ox x] =0.
\eean

Let $ \ox^n_{i=1}V_i$ be a tensor product of $\slt$-modules. 
The system of differential equations 
\bean\label{kz}
 \frac{\partial I}{\partial z_i} =
   \frac{1}{\kappa} \sum_{j \ne i}
   \frac{\Omega^{(i,j)}}{z_i-z_j} I, 
\qquad
\ i = 1, \dots , n,
\eean
on  a $ \ox^n_{i=1}V_i$-valued function
$I(z_1, \dots, z_n)$ is called the KZ equations.

\subsection{Irreducible $\slt$-modules}
For a nonnegative integer $i$ denote by $L_i$ the irreducible $i+1$-dimensional module with
basis $v_i, fv_i,\dots,f^iv_i$ and action $h.f^kv_i=(i-2k)f^kv_i$ for $k=0,\dots,i$; $f.f^kv_i=f^{k+1}v_i$ for $k=0,\dots, i-1$,
$f.f^iv_i=0$; $e.v_i=0$, $e.f^kv_i=k(i-k+1)f^{k-1}v_i$ for $k=1,\dots,i$.

For $m = (m_1,\dots,m_n) \in \Z^n_{\geq 0}$, denote
$|m| = m_1 + \dots + m_n$ and
$L^{\otimes m} =  L_{m_1}  \otimes  \dots  \otimes L_{m_n}$.  
 For
$J=(j_1,\dots,j_n) \in \mathbb{Z}_{\geq 0}^n$, with $j_s\leq m_s$ for $s=1,\dots,n$, the vectors
\bean
\label{fv}
f_Jv_m := f^{j_1}v_{m_1}\otimes \dots \otimes f^{j_n}v_{m_n} 
\eean
form a basis of
$L^{\otimes m}$. We have
\bea
f.f_Jv_m &=& \sum_{s=1}^n f_{J+1_s}v_m,
\qquad
h.f_Jv_m = ( |m|-2|J|) f_Jv_m,
\\
&&
e.f_Jv_m = \sum_{s=1}^n j_s (m_s-j_s+1) f_{J-1_s}v_m.
\eea
For $\lambda \in \Z$, introduce the weight subspace $L^{\otimes m}[\lambda] = \{ \ v \in L^{\otimes m} \
| \ h.v = \lambda v \}$ and the singular weight subspace
$\Sing L^{\otimes m}
[\lambda] = \{ \ v \in L^{\otimes m}[\lambda] \ | \ h.v = \lambda v, \
e.v = 0 \} $. We have the weight decomposition
$L^{\ox m} = \oplus_{k=0}^{|m|}L^{\ox m}[|m|-2k]$.
Denote
\bea
\Ik =\{ J\in \Z^n_{\geq 0}\ | \ |J|=k, \,j_s\leq m_s,\ s=1,\dots,n\}.
\eea
The vectors $(f_Jv)_{J\in\Ik}$ form a basis of $L^{\otimes m}[|m|-2k]$. 

\begin{rem}
The $\slt$-action on the sum of singular weight subspaces
$\Sing L^{\otimes m}[|m|-2k]$ generates the entire
$\slt$-module $L^{\otimes m}$. If $I(z_1,\dots,z_n)$ is an $L^{\ox m}$-valued 
solution of the KZ equations, then for any $x\in\slt$ the function 
$x.I(z_1,\dots,z_n)$ is also a solution, see \Ref{h invar}. These  observations
show that in order
to construct all  $L^{\ox m}$-valued solutions of the KZ equations it is enough to construct all
$\Sing L^{\otimes m}[|m|-2k]$-valued solutions for all $k$
and then generate the other  solutions by the $\slt$-action.

\end{rem}

\subsection
{Solutions of KZ equations with values in
 $\Sing L^{\otimes m} \big[ |m| - 2k \big]$ over $\C$}
  \label{Sol in C}

Given $k, n \in \Z_{>0}$, $m = ( m_1, \dots , m_n) \in \Z_{> 0}^n$,  $\ka\in\C^\times$,
denote $t=(t_1,\dots,t_k)$, $z=(z_1,\dots,z_n)$, 
define the  {\it master function}
\bean
\label{Master}
\Phi_{k,n,m}(t,z):
& = &
\Phi_{k,n,m} (t_1, \dots , t_k, z_1, \dots , z_n, \ka)
\\
\notag
& = 
&\prod_{i<j}   (z_i-z_j)^{m_im_j/2\ka} 
   \prod_{1 \leq i \leq j \leq k}  (t_i-t_j)^{2/\ka}
   \prod_{l=1}^{n} \prod_{i=1}^{k} (t_i-z_l)^{-m_l/\ka}.
\notag
\eean
For any function or differential form $F(t_1, \dots , t_k)$, denote
\bea
\Sym_t  [F(t_1, \dots , t_k)]  = \sum_{\sigma \in S_k}\!
F(t_{\sigma_1}, \dots , t_{\sigma_k}) ,
\ \ 
\Ant_t  [F(t_1, \dots , t_k)]  =  \sum_{\sigma \in S_k}\!
(-1)^{|\sigma|}F(t_{\sigma_1}, \dots , t_{\sigma_k}) .
\eea
For $J=(j_1,\dots,j_n)\in \Ik$ define the {\it weight function}
\bean
\label{WJ}
W_J(t,z)  = \frac{1}{j_1! \dots j_n!} 
                            \Sym_t \left[ \prod_{s=1}^{n}
                            \prod_{i=1}^{j_s}
                            \frac{1}{t_{ j_1 + \dots + j_{s-1}+i} - z_s}
                           \right]  .
\eean
For example,
\bea
&&
W_{(1,0,\dots,0)} =  \frac{1}{t_1 - z_1} ,
\qquad
W_{(2,0,\dots,0)}  =  \frac{1}{t_1 - z_1}  \frac{1}{t_2 - z_1} ,
\\
&&
\phantom{aaa}
W_{(1,1,0,\dots,0)} =   \frac{1}{t_1-z_1} \frac{1}{t_2-z_2}
+ \frac{1}{t_2-z_1}  \frac{1}{t_1-z_2}  .
\eea
The function
\bean
\label{vW}
W_{k,n,m}(t,z)=\sum_{J\in \Ik}W_J(t,z) f_Jv_m
\eean
is  the  $L^{\otimes m}[ |m| -2k]$-valued {\it vector weight function}.

Consider the  $L^{\otimes m}[ |m| -2k]$-valued
function
\bean
\label{intrep}
I^{(\gamma)}(z_1, \dots , z_n) =  \int_{\ga(z)}  \Phi_{k,n,m}(t,z, \ka)
    W_{k,n,m}(t,z) dt_1 \wedge \dots
   \wedge dt_k ,
\eean
where $ \gamma(z)$ in $\{z\} \times \C^k_t$ is a horizontal family of
$k$-dimensional cycles of the twisted homology defined by the multivalued
function $\Phi_{k,n,m}(t,z, m)$, see \cite{SV3,V1,V2}.
The cycles $\gamma(z)$ are multi-dimensional analogs of Pochhammer double loops.

\begin{thm}
\label{thm s}
The function  $I^{(\gamma)} (z)$ takes values in
 $ \Sing L^{\otimes m}[|m|-2k]$
and satisfies the KZ equations.
\end{thm}

This theorem and  its generalizations can be found, for example,  in \cite {CF,DJMM,SV1,SV2,SV3}.

The solutions in Theorem \ref{thm s}  are called the  {\it multidimensional 
hypergeometric solutions} of the KZ equations.
The coordinate functions
\bean
\label{chy}
I_J^{(\gamma)}(z_1, \dots , z_n)= \int_\gamma \Phi_{k,n,m}(t,z)
W_J(t,z) dt_1 \wedge \dots \wedge dt_k, \qquad J\in\Ik, 
\eean
 are called the {\it multidimensional hypergeometric functions}
 associated with
the master function $\Phi_{k,n,m}$. 

The fact that solutions in Theorem \ref{thm s} take values
in $ \Sing L^{\otimes m} [ |m|-2k ]$ may be reformulated as follows.
For any $J\in \mc I_{k-1}$,  we have
\bean
\label{Rel}
\sum_{s=1}^{n} (j_s + 1) (m_s - j_s) I_{J+ {\bf 1}_s}^{(\gamma)} (z)   =  0 ,
\eean
where we set $ I_{J+ {\bf 1}_s}^{(\gamma)} (z)=0$ if  $J+ {\bf 1}_s \notin \Ik$.

The pair consisting of the KZ equations \Ref{KZ} and hypergeometric solutions \Ref{s} 
is identified
with the pair consisting of the KZ equations \Ref{kz} and hypergeometric solutions \Ref{intrep}
with values in $\Sing L^{\ox m}[|m|-2]$. In this case the
system of equations in \Ref{Rel} is identified with
 equation \Ref{le}.

\subsection{Proof of Theorem \ref{thm s}}
\label{sec2.4}
We sketch the proof following \cite{SV3}. 
The reason to present a proof is to show later in Section \ref{Sol in F}
how a modification of this reasoning leads to a construction of polynomial solutions of 
the KZ equations over $\FF_p$.  

The proof  of Theorem \ref{thm s} is a generalization of the proof of Theorem \ref{thm1.1} and is
based on  cohomological relations.

It is convenient to reformulate the definition of the hypergeometric integral
\Ref{intrep}.
Given $k, n \in \Z_{>0}$ and a multi-index $J = (j_1, \dots , j_n)$ with $|J| \leq k$,
introduce a differential form
\bea
&&
\eta_J  =  \frac{1}{j_1 ! \cdots j_n !}
\operatorname{Ant}_t
\Big[
\frac{d(t_1 - z_1)}{t_1 -z_1} \wedge
\dots
\wedge \frac{d(t_{j_1} - z_1)}{t_{j_1} - z_1} \wedge
\frac{d(t_{j_1+1} - z_2)}{t_{j_1+1} - z_2} \wedge \dots
\\
&&
\phantom{aaaaasssaaaaaaaa}
 \wedge 
\frac{d(t_{j_1+ \dots + j_{n-1}+1} - z_n)}{t_{j_1+ \dots + j_{n-1}+1} - z_n}
\wedge \dots \wedge
\frac{d(t_{j_1+ \dots + j_{n}} - z_n)}{t_{j_1+ \dots + j_{n}} - z_n}
\Big]  ,
\eea
which is  a logarithmic differential form on $\C^n\times \C^k$
with coordinates $z,t$.
If $|J| = k$, then for any $z\in \C^n$ we have on $\{z\} \times \C^k$ the identity
\be
\eta_J  = W_J(t, z) dt_1  \wedge  \dots \wedge dt_k .
\ee
\begin{exmp}
For $k = n = 2$ we have
\bea
\eta_{(2,0)} \
& =&
\ \frac{d(t_1 - z_1)}{t_1 -z_1} \wedge \frac{d(t_2 - z_1)}{t_2 -z_1} ,
\\
\eta_{(1,1)} \
& = &
\ \frac{d(t_1 - z_1)}{t_1 -z_1} \wedge \frac{d(t_2 - z_2)}{t_2 -z_2} -
\frac{d(t_2 - z_1)}{t_2 -z_1} \wedge \frac{d(t_1 - z_2)}{t_1 -z_2} .
\eea
\end{exmp}
The hypergeometric integrals \Ref{intrep} can be  defined in terms of
the differential forms $\eta_J$:
\bean
\label{Iforms}
I^{(\gamma)} (z_1, \dots , z_n) =  \sum_{J\in \Ik} \ \Big(
\int_{\gamma (z)}\Phi_{k,n,m} \eta_J \Big)
 f_J v_m .
\eean

Introduce the logarithmic differential 1-forms
\bea
\alpha 
& =  &
\sum_{1\leq i<j\leq n}
 \frac{m_i m_j}{2\ka}
 \frac{d (z_i - z_j)}{z_i -z_j}  + 
\sum_{1\leq i<j\leq k} \frac2\ka
\frac{d (t_i - t_j)}{t_i -t_j} +
\sum_{s=1}^{ n} \sum_{i=1}^k \frac{-m_s}\ka
 \frac{d (t_i - z_s)}{t_i -z_s} ,
\\
\alpha'
& = & 
\sum_{1\leq i<j\leq k} \frac2\ka
\frac{d (t_i - t_j)}{t_i -t_j} +
\sum_{s=1}^{ n} \sum_{i=1}^k \frac{-m_s}\ka
 \frac{d (t_i - z_s)}{t_i -z_s} .
\eea
We shall use the following algebraic identities for logarithmic differential forms.

\begin{thm} [\cite{SV3}]
\label{ss2}
On $\C^n\times \C^k$ we have
\bean
\label{id1}
\alpha' \wedge \eta_J
= \sum_{s=1}^n (j_s +1) \frac{m_s - j_s}\ka 
 \eta_{J + {\bf 1}_s},  
\eean
for any $J$ with $|J|=k-1$, and
\bean
\label{id2}
\alpha \wedge   \sum_{J\in\Ik}  \eta_J f_Jv_m
    = \frac{1}{\kappa} \sum_{i<j} \Omega^{(i,j)} 
    \frac{d(z_i-z_j)}{z_i-z_j}  \wedge 
   \sum_{|J|=k}  \eta_J  f_Jv_m .
\eean
\end{thm}

\begin{proof}  Identity \Ref{id1} is the spacial case of Theorem 6.16.2 in \cite{SV3} for  the Lie algebra $\slt$.
Identity \Ref{id2} is  a special case of Theorem 7.5.2'' in \cite{SV3} for the Lie algebra $\slt$.
\end{proof}

\begin{cor}
\label{cor id}
On $\C^n\times \C^k$ we have
\bean
\label{id21}
\sum_{J\in\Ik} d (\Phi_{k,n,m}\eta_J)  f_Jv_m
     =  \frac{1}{\kappa} \sum_{i<j} \Omega^{(i,j)} 
    \frac{d(z_i-z_j)}{z_i-z_j} \wedge 
   \sum_{J\in\Ik} (\Phi_{k,n,m}\eta_J)  f_Jv_m, 
\eean
where the differential is taken with respect to variables $z,t$.
\end{cor}

Now we deduce from identity \Ref{id1} the following formula \Ref{id111}.  Since $|J|=k-1$, we can write
\bean
\label{t-part}
\eta_J =\sum_{l=1}^k c_{J,l}(t,z) dt_1\wedge\dots\wedge
\widehat{dt_l}\wedge\dots\wedge dt_k + \dots,
\eean
where the dots denote the terms having differentials  $dz_i$ and
$c_{J,l}(t,z)$ are rational functions of the form
\bean
\label{rat}
P_{J,l}(t,z)
\left(
\prod_{1\leq i<j\leq n}   (z_i-z_j)
   \prod_{1 \leq i < j \leq k}  (t_i-t_j)
   \prod_{l=1}^{n} \prod_{i=1}^{k} (t_i-z_l)
   \right)^{-1},
\eean
where $P_{J,l}(t,z)$ is a polynomial in $t,z$ with integer coefficients.
Also for any $s=1,\dots,n$ we have 
\bean
\label{tt-part}
\eta_{J + {\bf 1}_s} =
W_{J + {\bf 1}_s} dt_1\wedge\dots\wedge dt_k + \dots,
\eean
where the dots denote the terms having differentials  $dz_i$. 
Formula  \Ref{id1} implies that for any $J$ with $|J|=k-1$ we have 
the identity
\bean
\label{id111}
&&
d_t\left( \Phi_{k,n,m}\sum_{l=1}^k c_{J,l}(t,z) dt_1\wedge\dots\wedge
\widehat{dt_l}\wedge\dots\wedge dt_k  \right)
\\
\notag
&&
\phantom{aaaaaa}
= \ \sum_{s=1}^n  (j_s +1) \frac{m_s - j_s}\ka 
\Phi_{k,n,m}W_{J + {\bf 1}_s} dt_1\wedge\dots\wedge dt_k,
\eean
where $d_t$ denotes the differential with respect to the variables $t$.

Now we deduce from identity \Ref{id21} the following formula \Ref{der zi}. 
Fix $i\in\{1,\dots,n\}$. For any $J\in\Ik$, write
\bean
\label{dzi}
&&
\Phi_{k,n,m}\eta_J = \Phi_{k,n,m} W_J dt_1\wedge\dots\wedge dt_k
\\
\notag
&&\phantom{aaa}
+
dz_i\wedge
\left(\Phi_{k,n,m} \sum_{l=1}^k c_{J,i,l}(t,z) dt_1\wedge\dots\wedge
\widehat{dt_l}\wedge\dots\wedge dt_k \right)
+ \dots,
\eean
where the dots denote the terms which contain $dz_j$ with $j\ne i$, 
and the coefficients $c_{J,i,l}(t,z)$ are rational functions in $t,z$ of the form
\bean
\label{rat2}
P_{J,i,l}(t,z)
\left(
\prod_{1\leq i<j\leq n}   (z_i-z_j)
   \prod_{1 \leq i < j \leq k}  (t_i-t_j)
   \prod_{l=1}^{n} \prod_{i=1}^{k} (t_i-z_l)
   \right)^{-1},
\eean
where $P_{J,i,l}(t,z)$
is a polynomial in $t,z$ with integer coefficients.

Formula  \Ref{id21} implies  that for any $i\in\{1,\dots,n\}$ we have
\bean
\label{der zi}
&&
\sum_{J\in\Ik}
 \Big(
 \frac{\der}{\der z_i}
 \left(
\Phi_{k,n,m}W_J  \right) dt_1\wedge\dots\wedge dt_k
\\
\notag
&&
\phantom{aa}
 +
d_t\big(\Phi_{k,n,m}\sum_{l=1}^n c_{J,i,l}(t,z) dt_1\wedge\dots\wedge
\widehat{dt_l}\wedge\dots\wedge dt_k\big)
 \Big)f_Jv_m
\\
\notag
&&
\phantom{aaaa}
     =  \frac{1}{\kappa} \sum_{j\ne i} \frac{\Omega^{(i,j)}}{z_i-z_j}  
   \sum_{J\in\Ik} \Phi_{k,n,m}W_J  dt_1\wedge\dots\wedge dt_k  f_Jv_m, 
\eean
where $d_t$ denotes the differential with respect to the variables $t$.

Integrating both sides of equations  \Ref{id111} and \Ref{der zi} over $\ga(z)$ and using Stokes' theorem we obtain equations
\Ref{Rel} and \Ref{kz} for the vector $I^{(\ga)}(z)$ in \Ref{intrep}. Theorem \ref{thm s} is proved.

\subsection
{Solutions of KZ equations with values in
 $\Sing L^{\otimes m} \big[ |m| - 2k \big]$ over $\FF_p$}
  \label{Sol in F}

Given $k, n \in \Z_{>0}$,  $m = ( m_1, \dots , m_n) \in \Z_{>0}^n$,  $\ka\in\Q^\times$,
let $p>2$ be a prime number such that $p$ does not divide the numerator of $\ka$.
  In this case equations \Ref{Rel} and \Ref{kz} are well-defined over the field $\FF_p$
and we may discuss their polynomial solutions in $\FF_p[z_1,\dots,z_n]$.

\medskip
  Choose positive integers $M_s$ for $s=1,\dots,n$, 
$M_{i,j}$ for $1\le i<j\leq n$, and $M^0$, such that 
\bea
M_s\equiv -\frac{m_s}\ka,
\qquad
M_{i,j}\equiv \frac{m_im_j}{2\ka},
\qquad
M^0\equiv \frac2\ka 
\qquad (\on{mod}\,p).
\eea
      Fix integers $q=(q_1,\dots,q_k)$. 
Let $t=(t_1,\dots,t_k)$, $z=(z_1,\dots,z_n)$ be variables.
Define the  {\it master polynomial}
\bean
\label{Mast P}
\Phi^{(p)}_{k,n,M}(t,z):
& = &
\Phi^{(p)}_{k,n,M} (t_1, \dots , t_k, z_1, \dots , z_n)
\\
\notag
& = 
&\prod_{1\leq i<j\leq n}   (z_i-z_j)^{M_{i,j}} 
   \prod_{1 \leq i \leq j \leq k}  (t_i-t_j)^{M^0}
   \prod_{s=1}^{n} \prod_{i=1}^{k} (t_i-z_s)^{M_s}.
\notag
\eean
Consider 
 the Taylor expansion of the vector 
\bean
\label{Pvec}
 \sum_{J\in\Ik} \Phi^{(p)}_{k,n,M}(t,z)W_J(t,z) f_Jv_m =
  \sum_{i_1,\dots,i_k} \bar I^{(i_1,\dots,i_k)}(z,q)(t_1-q_1)^{i_1}\dots(t_k-q_k)^{i_k}.
\eean
Notice that each coordinate $\Phi^{(p)}_{k,n,M}(t,z)W_J(t,z)$ is a polynomial in $t,z$ 
with integer coefficients due to the positivity of
the integers $M_s, M_{i,j},M^0$ and the definition of the weight functions $W_J(t,z)$. 
Hence the Taylor coefficients $\bar I^{(i_1,\dots,i_k)}(z,q)$
 are vectors of polynomials in $z$ with integer coefficients.
Let 
$I^{(i_1,\dots,i_k)}(z,q) \in (\FF_p[z])^{\dim L^{\ox m}[|m|-2k]} $ be their canonical projection modulo $p$.

\begin{thm} 
\label{thm 2.4}
For any integers $q=(q_1,\dots,q_k)$ and positive integers $l=(l_1,\dots,l_k)$, the vector of polynomials 
$I(z,q):= I^{(l_1p-1,\dots,l_kp-1)}(z,q)$ satisfies equations \Ref{Rel} and \Ref{kz}.

\end{thm}

The parameters $q$, $l_1p-1,\dots,l_kp-1$ are analogs of cycles $\ga$ in Section \ref{Sol in C}.

\begin{proof}

To prove that $I^{(l_1p-1,\dots,l_kp-1)}(z,q)$ satisfies \Ref{Rel} 
and \Ref{kz} consider the Taylor expansions at $(t_1,\dots,t_k)=(q_1,\dots,q_k)$ of both sides of equations
\Ref{id111} and 
\Ref{der zi}, divide them by $dt_1\wedge\dots\wedge dt_k$.
Notice that the Taylor expansions are well defined due to formulas \Ref{rat} and \Ref{rat2}. 

Project  the  Taylor coefficients of $(t_1-q_1)^{l_1p-1}\dots (t_k-q_k)^{l_kp-1}$
  to $(\FF_p[z])^{\dim L^{\ox m}[|m|-2k]}$. 
 Then  the terms coming from
the $d_t(\,)$-summands  equal 
zero since $d(t_i^{l_ip})/dt_i=l_ipt_i^{l_ip-1}\equiv 0$ (mod $p$), and we obtain equations
\Ref{Rel} 
and \Ref{kz}.
\end{proof}

\begin{exmp}
\label{ex2.5}

Let $p=3$,  $\ka=4$, $n=5$, $k=2$, $m_1=\dots=m_5=1$. 
Notice that in this case $\ka\equiv 1$ (mod $3$) and we may set $\ka=1$.

The set $\Ik$ consists of ten elements $J=(j_1,\dots,j_5)$ with $j_i\in\{0,1\}$ and $j_1+\dots+j_5=2$.
The space
$L^{\ox m}[|m|-2k]=(L_1)^{\ox 5}[1]$ 
 has basis
$f_Jv_m=f^{j_1}v_1 \ox \dots \ox f^{j_5}v_1$, $J\in \Ik$. 
We have 
\bea
\Om^{(1,2)} v_1\ox v_1\wedge\dots
&\equiv&
 -v_1\ox v_1\wedge\dots,
\\
\Om^{(1,2)} fv_1\ox fv_1\wedge\dots
&\equiv&
 -fv_1\ox fv_1\wedge\dots,
\\
\Om^{(1,2)} fv_1\ox v_1\wedge\dots
&\equiv&
 fv_1\ox v_1\wedge\dots + v_1\ox fv_1\wedge\dots,
\\
\Om^{(1,2)} v_1\ox fv_1\wedge\dots
&\equiv&
 fv_1\ox v_1\wedge\dots + v_1\ox fv_1\wedge\dots
 \eea
 (mod $3$).
 The other $\Om^{(i,j)}$ act similarly.
The system of equations \Ref{Rel}  on 
\\
$I(z)=\sum_{J\in\Ik}I_J(z)f_Jv_m$  consists of five equations. The first is
\bea
I_{(1,1,0,0,0)}(z) + I_{(1,0,1,0,0)}(z)+ I_{(1,0,0,1,0)}(z)+I_{(1,0,0,0,1)}(z)\equiv 0\ \ (\on{mod}\,3),
\eea
where $z=(z_1,\dots,z_5)$,  the other  are  similar. 
Let $t=(t_1,t_2)$. 
 We may choose the master polynomial 
\bea
\Phi^{(p=3)}_{2,5,M}(t,z) = (t_1-t_2)^2 \prod_{1\leq i<j\leq 5}(z_i-z_j)^2\prod_{i=1}^2\prod_{s=1}^5(t_i-z_s)^2.
\eea
   Fix  integers  $q=(0,0)$ and $l=(4,3)$. Then the vector
   \bean
   I^{(11,8)}(z) =\sum_{J\in \Ik} I^{(11,8)}_J(z) f_Jv_m
   \eean
with
\bean
\label{11}
I^{(11,8)}_{(1,1,0,0,0)}(z)=-z_3-z_4-z_5, \quad  I^{(11,8)}_{(1,0,1,0,0)}(z)=-z_2-z_4-z_5,
\eean
and similar other coordinates satisfies equations \Ref{Rel} and \Ref{kz}.
     \end{exmp}

\begin{exmp}
\label{extwo}

Let  $\ka=4$, $n=2$, $k=2$, $m_1=m_2=2$.
The space $L_2^{\ox 2}[0]$ has basis $f^2v_2\ox v_2$, $fv_2\ox fv_2$, $v_2\ox f^2v_2$. 
The system of equations \Ref{Rel}   takes the form:
\bea
I_{(2,0)}(z)+I_{(1,1)}(z) = 0,\quad
I_{(1,1)}(z) +  I_{(0,2)}(z) = 0.
\eea
Let $p=4l+3$ for some $l$. We may choose
\bea
\Phi^{(p)}_{2,2,M}(t_1,t_2,z_1,z_2) = (z_1-z_2)^{\frac{p+1}2}(t_1-t_2)^{\frac{p+1}2}\prod_{i=1}^2\prod_{s=1}^2(t_i-z_s)^{\frac{p-1}2}.
\eea
Notice that $\frac{p+1}2$ is even, the polynomial $\Phi^{(p)}_{2,2,M}(t_1,t_2,z_1,z_2) $ is symmetric with respect to permutation
of $t_1,t_2$, and the solution 
\bean
\label{s-s}
&&
\phantom{aaa}
I^{(p-1,p-1)}(z_1,z_2) 
\\
\notag
&&
\phantom{aa}
=
 (z_1-z_2)^{\frac{p+1}2}
 \!
 \big(c_{(2,0)}(z_1,z_2)
f^2v_2\ox v_2 + c_{(1,1)}(z_1,z_2)fv_2\ox fv_2+c_{(2,0)}(z_1,z_2)v_2\ox f^2v_2)
\eean
 is nonzero. Here $c_J(z_1,z_2)$ are the polynomials determined by the construction of Section  \ref{Sol in F}.

For example, for $p=3$, 
   \bean
   \label {2.2.2}
  I^{(2,2)}(z) = (z_1-z_2)^2(f^2v_2\ox v_2-fv_2\ox fv_2+v_2\ox f^2v_2).
   \eean

\end{exmp}

\subsection{Relation of solutions to integrals over $\FF_p^k$}

For a polynomial $F(t_1,\dots,t_k)$ $\in\FF_p[t_1$, \dots, $t_k]$ and a subset $\ga\subset \FF^k_p$
define the integral 
\bea
\int_{\ga} F(t_1,\dots,t_k) := \sum_{(t_1,\dots, t_k)\in\ga}F(t_1,\dots,t_k).
\eea

\begin{thm}
\label{thm 2int}  Fix  $x_1,\dots,x_n\in \FF_p$.
Consider the vector of polynomials 
\bea
F(t):=\Phi^{(p)}_{k,n,M}(t_1,\dots,t_k,x_1,\dots,x_n)\sum_{J\in \Ik} W_J(t_1,t_2,x_1,\dots,x_n) f_Jv_m,
\eea
 of formula \Ref{Pvec}.  Assume that
$\deg_{t_i} F(t_1,\dots,t_k) < 2p-2$ for $i=1,\dots,k$. Consider the solution $I^{(p-1,\dots,p-1)}(z,q)$ of equations \Ref{Rel} and \Ref{kz}, described in Theorem \ref{thm 2.4}. Then 
\bean
\label{2INT}
I^{(p-1,\dots,p-1)}(x_1,\dots,x_n,q) = (-1)^k \int_{\FF_p^k}F(t_1,\dots,t_k) .
\eean
\end{thm}
This integral is a $p$-analog of the  hypergeometric integral \Ref{chy}.

\begin{proof} Theorem \ref{thm 2int} is a straightforward corollary of formula \Ref{sum=}, cf. the proof of Theorem
\ref{thm int}. 
\end{proof}

\begin{exmp}
The polynomial $F(t_1,t_2)$ of Example \ref{extwo} satisfies the inequalities
\\
 $\deg_{t_i} F(t_1,t_2) < 2p-2$ for $i=1,2$.
\end{exmp}

\subsection{Example of a $p$-analog of skew-symmetry}
For $J\in\Ik$, the differential forms $W_J(t,z)dt_1\wedge\dots\wedge dt_k$
are skew-symmetric with respect to permutations of $t_1,\dots,t_k$. Here is an
example of a $p$-analog of that skew-symmetry. Another demonstration of the skew-symmetry see in
 Example \ref{ex231}.

Consider the KZ differential equations with parameters $n$, $k$, $\ka$,  $m_1,\dots,m_n \in \Z_{>0}$,
where $\ka$,  $m_1,\dots,m_n$
are even,
$\ka=2\ka',  m_1=2m_1',\dots,m_n=2m_n'$. Assume  that $\ka'$ is even and the 
prime $p$ is such that $\ka'\big|(p-1)$ and $(p-1)/\ka'$ is odd, cf.  Example \ref{ex231}.
 Choose
\bean
\label{s-ma}
&&
\Phi^{(p)}_{k,n,M}(t,z) =\prod_{1\leq i<j\leq n} (z_i-z_j)^{M_{i,j}}
\prod_{1\leq i<j\leq k} (t_i-t_j)^{p-\frac{p-1}{\ka'}}
\prod_{i=1}^k\prod_{s=1}^n(t_i-z_s)^{m_s'\frac{p-1}{\ka'}}
\\
\notag
&&
\phantom{aa}
=\prod_{1\leq i<j\leq n} (z_i-z_j)^{M_{i,j}}
\!
\left( 
\prod_{1\leq i<j\leq k} (t_i-t_j)^{\ka'-1}
\prod_{i=1}^k\prod_{s=1}^n(t_i-z_s)^{m_s'}
\right)^{\frac{p-1}{\ka'}}
\!\!\!
\prod_{1\leq i<j\leq k} (t_i-t_j).
\eean
Notice that 
\bean
\label{tphi}
\phi(t,z):=\prod_{1\leq i<j\leq k} (t_i-t_j)^{\ka'-1}
\prod_{i=1}^k\prod_{s=1}^n(t_i-z_s)^{m_s'}
\eean
as well as the
product $\prod_{1\leq i<j\leq k} (t_i-t_j)$
are skew-symmetric with respect to permutations of $t_1,\dots,t_k$.

Let $a$ be a generator of the cyclic group $\FF_p^\times$. Let $x=(x_1,\dots,x_n)\in \FF_p^n$. For $\ell =1,\dots, \ka'$, denote
\bean
\label{ga-el}
\phantom{aaa}
\ga_\ell(x)=\{t\in\FF^k_p\ |\ \phi(t,x)^{\frac{p-1}{\ka'}} = a^{\ell\frac{p-1}{\ka'}}\},
\qquad
\ga_0(x)=\{t\in\FF^k_p\ |\ \phi(t,x)=0\}.
\eean
The partition of $\FF_p^k$ by subsets  $(\ga_\ell(x))_{\ell=0}^{\ka'}$ is invariant with respect to the action of the symmetric group
$S_k$ of permutations of $t_1,\dots,t_k$. For  every $\ell$, the subset  $\ga_\ell(x)$ is invariant with respect 
to the action of the alternating subgroup
$A_k\subset S_k$. For $J\in\Ik$ the restriction of the function $W_J(t,x)\prod_{1\leq i<j\leq k} (t_i-t_j)$ 
 to the set $\ga_\ell(x)$ is $A_k$-invariant.
We have
\bea
\int_{\FF^k_p} \Phi^{(p)}_{k,n,M}(t,z) W_J(t,x)
=\prod_{1\leq i<j\leq n} (z_i-z_j)^{M_{i,j}}\
\sum_{\ell=1}^{\ka'/2} \ 2 a^{\ell\frac{p-1}{\ka'}}
\int_{\ga_\ell(x)} W_J(t,x)\prod_{1\leq i<j\leq k} (t_i-t_j).
\eea

\subsection{Relation of solutions to surfaces over $\FF_p$}
\label{RSSF}

\begin{exmp}
\label{ex231}

For distinct $x_1,x_2 \in\FF_p$ let
$\Ga(x_1,x_2)$ be the closure in $P^1(\FF_p)\times P^1(\FF_p)$ of 
the affine surface 
\bean
\label{ec}
y^2 = (t_1-t_2)(t_1-x_1)(t_2-x_1)(t_1-x_2)(t_2-x_2) ,
\eean
where  $P^1(\FF_p)$ is the projective line over $\FF_p$. 
For a rational function $h : \Ga(x_1,x_2) \to \FF_p$  define the integral  
\bean
\label{int}
\int_{\Ga(x_1,x_2)} h = {\sum_{P\in\Ga}}' h(P),
\eean
as the sum over all points $P\in\Ga(x_1,x_2)$, where $h(P)$  is defined.

Recall
\bea
&&
W_{(2,0)}(t_1,t_2,x_1,x_2)  =  \frac{1}{t_1 - x_1}  \frac{1}{t_2 - x_1} ,
\qquad
W_{(0,2)}(t_1,t_2,x_1,x_2)  =  \frac{1}{t_1 - x_2}  \frac{1}{t_2 - x_2} ,
\\
&&
\phantom{aaa}
W_{(1,1)}(t_1,t_2,x_1,x_2) =   \frac{1}{t_1-x_1} \frac{1}{t_2-x_2}
+ \frac{1}{t_2-x_1}  \frac{1}{t_1-x_2}  .
\eea

\begin{thm}
\label{thm twopts} 
Let $p=4l+3$ for some $l$.
Let 
\bea
c_{(2,0)}(z_1,z_2)
f^2v_2\ox v_2 + c_{(1,1)}(z_1,z_2)fv_2\ox fv_2+c_{(2,0)}(z_1,z_2)v_2\ox f^2v_2
\eea
be the vector of polynomials appearing in the solution 
\Ref{s-s} of the KZ equations of Example \ref{extwo}. Then 
\bean
\label{2=sol} 
   c_{(2,0)}(x_1,x_2)
&= &
\int_{\Ga(x_1,x_2)}\frac{t_1-t_2}{(t_1 - x_1)(t_2 - x_1)},
\\
\notag
  c_{(1,1)}(x_1,x_2)
& = & 
\int_{\Ga(x_1,x_2)}\frac{t_1-t_2}{(t_1 - x_1)(t_2 - x_2)}
+
\int_{\Ga(x_1,x_2)}\frac{t_1-t_2}{(t_2 - x_1)(t_1 - x_2)},
 \\
\notag
   c_{(0,2)}(x_1,x_2)
& = & 
\int_{\Ga(x_1,x_2)}\frac{t_1-t_2}{(t_1 - x_2)(t_2 - x_2)}.
 \eean

\end{thm}

\begin{proof} 
The values of $W_J(t_1,t_2,x_1,x_2)$ at infinite points of $\Ga(x_1,x_2)$ equal zero, so the integrals are sums
over points of
the affine surface. We prove the first equality in \Ref{2=sol}. We have
\bea
&&
\int_{\Ga(x_1,x_2)}\frac{t_1-t_2}{(t_1 - x_1)(t_2 - x_1)}
 = 
\sum_{t_1,t_2 \ne x_1} \frac{t_1-t_2}{(t_1 - x_1)(t_2 - x_1)}
\\
&&
\phantom{a}
+\sum_{t_1,t_2} \frac{t_1-t_2}{(t_1 - x_1)(t_2 - x_1)}
\Big((t_1-t_2)\prod_{i=1}^2\prod_{s=1}^2(t_i-x_s)\Big)^{\frac{p-1}2}
\\
&&
\phantom{aa}
=  \sum_{t_1,t_2\in\FF_p} [(t_1-x_2)^{p-2}-(t_1-x_1)^{p-2}] 
 + \sum_{t_1,t_2\in\FF_p}  \sum_{i_1,i_2}  c^{i_1,i_2}(x_1,x_2) t_1^{i_1}t_2^{i_2}
=    c_{(2,0)}(x_1,x_2). 
\eea
\end{proof}

\begin{rem}
Consider the projection $\Ga(x_1,x_2) \to \FF_p^2$, $(t_1,t_2,y)\mapsto (t_1,t_2)$.
For any distinct $t_1,t_2\in\FF_p$ exactly one of the two points $(t_1,t_2)$, $(t_2,t_1)$ lies in the 
image of the projection, since
$(t_1-t_2)(t_1-x_1)(t_2-x_1)(t_1-x_2)(t_2-x_2)$ is skew-symmetric in $t_1,t_2$ and 
$-1$ is not a square if $p=4l+3$.

\end{rem}
\end{exmp}

\section{Resonances in $\slt$ KZ equations}
\label{3}

\subsection{Resonances in conformal field theory over $\C$}
\label{kzC}

Let $m_1, \ldots , m_n, k \in \Z_{>0}$,
$L^{\otimes m}=L_{m_1}\ox\dots\ox L_{m_n}$.
 Assume that $\kappa>2$ is an integer.
Assume that
$$
0\leq  m_1,  \dots ,  m_n,  m_1 + \dots + m_n - 2k  \leq \kappa - 2 .
$$
Consider the positive integer
\bean
\label{RES}
\ell  = \ka  - 1 - |m| +2k .
\eean
For $z=(z_1,\dots,z_n)\in \C^n$ with distinct coordinates define
$$
B_{k,n,m}(z) = \{ w \in L^{\otimes m}\ |\  h.w = (|m|-2k)w, \ e.w = 0, \
(ze)^\ell w = 0 \},
$$
where $ze : L^{\otimes m} \to L^{\otimes m}$ is the linear operator  defined by the formula
$$
w_1 \ox\dots\ox w_n\ \mapsto\ \sum_{s=1}^n z_s w_1\ox\dots\ox ew_s \ox\dots\ox w_n ,
$$
for any $w_1 \ox\dots\ox w_n \in L^{\otimes m}$. This vector
space is called the {\it  space of conformal blocks}.

\begin{exmp}
\label{ex3.1}
  Let $k=1$, $|m|=\ka$, $\ell=1$, Then
\bea
B_{k,n,m}(z) = \Big\{ \sum_{s=1}^n I_s v_{m_1}\ox\dots\ox fv_{m_s}\ox\dots\ox v_{m_n}\  \Big| \
\sum_{s=1}^nm_sI_s=0,\ 
\sum_{s=1}^nz_sm_sI_s=0\Big\} .
\eea

\end{exmp}

\begin{thm}[\cite{FSV1, FSV2}]
The family of subspaces 
\bea
B_{k,n,m}(z) \subset 
\Sing L^{\otimes m}[|m|-2k],
\eea
depending on $z$, is  invariant with respect to the
KZ equations.
\qed
\end{thm}

\begin{thm}
[\cite{FSV1, FSV2}]
\label{thm r}
All the hypergeometric solutions of the KZ equations with values in
$\Sing L^{\otimes m}[|m|-2k]$, constructed in Section \ref{Sol in C},
take values in the subspaces of conformal blocks.
\end{thm}

\begin{proof} Theorem \ref{thm r}  is proved in \cite{FSV1}. Another proof for arbitrary simple Lie algebras
is given in \cite{FSV2}.  Let $I^{(\ga)}(z) =\sum_{J\in\Ik} 
I^{(\ga)}_J(z) F_Jv_m$ be a hypergeometric solution. We need to check 
that $(ze)^\ell I^{(\ga)}(z) = 0$. This equation is a system of algebraic equations on the coefficients 
$(I^{(\ga)}_J(z))_{J\in\Ik}$. The equations of the system  are labeled by basis vectors of
$L^{\ox m}[|m| - 2(k-\ell)]$. Namely, for any $Q\in \mc I_{k-\ell}$ one calculates the coefficient of
$F_Qv_m$ in  $(ze)^\ell I^{(\ga)}(z)$ and equate that coefficient to zero, cf. the second equation in  Example 
\ref{ex3.1}.
Such an equation follows  from a cohomological relation. Namely, the 
corresponding differential $k$-form, whose integral over $\ga(z)$
has to be zero, equals the differential with respect to the $t$-variables of some differential $k-1$-form $\eta_{n,k,\ell,Q}(t,z)$.
Then the desired equation holds by Stokes' theorem, see this reasoning on pages 182--184 in \cite{FSV1}.  
This proves Theorem \ref{thm r}.
\end{proof}

\begin{rem}
 That $k-1$-form $\eta_{n,k,\ell,Q}(t,z)$ is determined by the numbers $n, k, \ell$ and the index $Q$ and has the form
\bean
\label{mu}
\phantom{aaaaaa}
\eta_{n,k,\ell,Q}(t,z)=
\frac{\Phi_{k,n,m}(t,z)}
{\prod_{1\leq i<j\leq n}   (z_i-z_j)
   \prod_{1 \leq i < j \leq k}  (t_i-t_j)
\prod_{i=1}^k\prod_{s=1}^n(t_i-z_s)}
\mu_{n,k,\ell,Q}(t,z),
\eean
where
$\mu_{n,k,\ell,Q}(t,z)$ is a polynomial differential $k-1$-form in $t,z$ with integer coefficients
determined by ${n,k,\ell,Q}$ only, see  pages 182--184  in \cite{FSV1}.

\end{rem}

\subsection{Resonances over $\FF_p$}

Given $k, n \in \Z_{>0}$,  $m = ( m_1, \dots , m_n) \in \Z_{>0}^n$,  $\ka\in\Z_{>0}$,
let $p>2$ be a prime number such that $p$ does not divide  $\ka$.
  Choose positive integers $M_s$ for $s=1,\dots,n$, 
$M_{i,j}$ for $1\le i<j\leq n$,  $M^0$ and $K$  such that 
\bea
M_s\equiv -\frac{m_s}\ka,
\quad
M_{i,j}\equiv \frac{m_im_j}{2\ka},
\quad
M^0\equiv \frac2\ka, \quad K\equiv \frac1\ka
\qquad
(\on{mod}\,p).
\eea
  Fix  integers $q=(q_1,\dots,q_k)$.  As in Section \ref{Sol in F} for any nonnegative integers $l_1,\dots,l_k$
  define the vector
  $I^{(i_1,\dots,i_k)}(z,q) \in (\FF_p[z])^{\dim L^{\ox m}[|m|-2k]}$.

\begin{thm} 
\label{thm 3}
Let $\ell\in\Z_{>0}$ be such that
\bean
\label{R}
(\ell-1)K - \sum_{s=1}^nM_s - (k-1)M^0 \equiv 1 \quad (\on{mod}\, p).
\eean
Then for any integers $q=(q_1,\dots,q_k)$ and positive integers $l=(l_1,\dots,l_k)$, the vector of polynomials 
$I^{(l_1p-1,\dots,l_kp-1)}(z,q)$ satisfies the equation
\bean
\label{RE}
(ze)^\ell I^{(l_1p-1,\dots,l_kp-1)}(z,q) = 0.
\eean

\end{thm}

\begin{rem}
The resonance equation \Ref{RES} has the form
\bea
\frac{\ell-1}\ka=1 - \frac{|m|}\ka +\frac2\ka(k-1).
\eea
Equation \Ref{R} is the reduction modulo $p$ of that equation.

\end{rem}

\begin{proof}
The proof of Theorem \ref{thm 3} is similar to the proof of Theorem \ref{thm s} and uses the universal differential $k-1$-forms 
$\eta_{n,k,\ell,Q}(t,z)$ of Section  \ref{kzC} instead of the differential $k-1$-forms $\eta_J(t,z)$ in \Ref{t-part}.
\end{proof}

\begin{exmp}

Let $p=3$,  $\ka=4$, $n=5$, $k=2$, $m_1=\dots=m_5=1$.
Consider the  vector $I^{(11,8)}(z) =\sum_{J\in \Ik} I^{(11,8)}_J(z) f_Jv_m$ of  Example \ref{ex2.5}, 
which is a
solution of \Ref{kz} and \Ref{Rel}. The resonance equation \Ref{R} 
in this case takes the form $\ell+1\equiv 0$ (mod $3$)
and is satisfied for $\ell=2$. The condition $(ze)^2I^{(11,8)}(z)=0$ means
\bean
\label{RRR}
\sum_{J=(j_1,\dots,j_5)\in\Ik} I^{(11,8)}_J(z) \prod_{i=1;\, j_i=1}^5\!z_i
\ \equiv\ 0\quad (\on{mod}\,3).
\eean
Equation \Ref{RRR} takes the form
\bea
-z_1z_2(z_3+z_4+z_5)-\dots - z_4z_5(z_1+z_2+z_3)= -3\sum_{1\leq i<j<k\leq 5} z_iz_jz_k \equiv 0
\eea
(mod $3$).

  \end{exmp}

\section{KZ equations over $\FF_p$ for other Lie algebras}
\label{4}

The KZ equations are defined for any simple Lie algebra $\g$ or more generally for any Kac-Moody algebra, see for example
\cite{SV3}. Similarly to what was done in Sections \ref{2} and \ref{3}, one can construct polynomial solutions of those KZ equations over 
$\FF_p$
as well as of the singular vector equations and resonance equations over $\FF_p$.

 The construction of the polynomial solutions over $\FF_p$ in the $\slt$ case was based on 
the algebraic identities for logarithmic differential forms \Ref{id1}, \Ref{id2} and the associated cohomological relations 
\Ref{id111}, \Ref{der zi} as well as on the cohomological relations associated with the differential forms
$\eta_{n,k,\ell,K}(t,z)$ in \Ref{mu}. For an arbitrary Kac-Moody algebra the analogs of the 
 algebraic identities in  \Ref{id1} and \Ref{id2} are the identities  of 
 Theorems 6.16.2  and 7.5.2'' in \cite{SV3}, respectively. For an arbitrary simple Lie algebra,
 the construction of analogs of the cohomological identities for the differential forms $\eta_{n,k,\ell,K}(t,z)$ 
 is the main result of   \cite{FSV2}.

\begin{rem}
The $\FF_p$-analogs of multidimensional hypergeometric integrals associated with arrangements of hyperplanes see in
\cite{V4}.  Remarks on the Gaudin model and Bethe ansatz over $\FF_p$ see in \cite{V3}.

\end{rem}

\end{document}